\documentclass[12pt,centertags]{amsart}
\usepackage{amsmath,amstext,amsthm,a4,amssymb,amscd}
\usepackage[mathscr]{eucal}
\usepackage{mathrsfs}
\usepackage{epsf}
\numberwithin{equation}{section}

\newtheorem{thm}{Theorem}[section]
\newtheorem{lemma}[thm]{Lemma}
\newtheorem{prop}[thm]{Proposition}

\theoremstyle{definition}

\theoremstyle{definition}
\newtheorem{defn}[thm]{Definition}
\newcommand{\be}{\begin{eqnarray}}
\newcommand{\ee}{\end{eqnarray}}

\newcommand{\comment}[1]{}

\begin{document}

\title{Equivariant Morse inequalities and applications}
\author{Wen Lu}

\address{Mathematisches Institut,
Universit\"at zu K\"oln, Weyertal 86-90, 50931 K\"oln, Germany}
\email{wlu@math.uni-koeln.de}

\date{\today}

\begin{abstract}
In this paper, we prove equivariant Morse inequalities via Bismut-Lebeau's
analytic localization techniques. As an application, we obtain Morse inequalities on compact
manifold with nonempty boundary by applying equivariant Morse inequalities to the doubling manifold.
\end{abstract}

\maketitle

\section{Introduction}\label{s5}

In his influential work \cite{Witten82}, Witten sketched analytic proofs of the degenerate
Morse inequalities of Bott \cite{Bott54} for Morse functions whose critical submanifolds
are nondegenerate in the sense of Bott.
Rigorous proofs were given by Bismut \cite{Bismut86}, by using heat kernel methods,
and later by Helffer and Sj\"ostrand \cite{Hellfer88}, by means of semiclassical analysis.
Braverman and Farber \cite{Braverman97} provided another proof using the Witten
deformation techniques suggested by Bismut \cite{Bismut86}.

Concerning the standard Morse inequalities (i.e., for Morse functions with
isolated critical points), an analytic proof is given by Zhang \cite[Chap.\,5]{Zhang01}, in the spirit
of the analytic localization techniques developed by Bismut-Lebeau \cite[\S 8-9]{Bismut74}.
Moreover, \cite[Chap.\,6]{Zhang01} contains a complete proof of the isomorphism between
the Thom-Smale complex and the Witten instanton complex. Following the ideas
in \cite{Zhang01}, we give here a proof of degenerate Morse inequalities by similar techniques.

Let us mention the related papers \cite{Braverman97,Braverman06,Feng01}.
In \cite{Braverman97, Braverman06}, Braverman, Farber and Silantyev
used Witten deformation techniques to study the Novikov number
associated to closed differential 1-forms nondegenerate in the sense of Bott and Kirwan, respectively.
In this way, they obtained
Novikov-type inequalities associated to a closed differential 1-form.
When the closed differential form is exact, these inequalities turn
to Morse inequalities.

In \cite{Feng01}, Feng and Guo establish Nivokov's type inequalities
associated to vector fields instead of closed differential forms
under a natural assumption on the zero-set of the vector field.

In this paper, we work out equivariant Morse-Bott inequalities along the lines of
\cite{Zhang01} (cf. \cite[\S8-9]{Bismut74}).
Compared to \cite{Feng01}, where Bismut-Lebeau's analytic localization techniques are applied along
the lines of \cite{Zhang01}, we can choose the geometrical data near
the singular points as simple as possible,
due to the equivariant Morse's Lemma \cite{Wasserman69}. As an application,
we get degenerate Morse inequalities for manifolds with nonempty boundary
by passing to the doubling manifold. Thus, we extend the result from
\cite{Zadeh10} to the most general situation.

Let $M$ be a smooth $m$-dimensional closed and connected
manifold, and let $G$ be a finite group acting smoothly on $M$.
Let $f: M\rightarrow \mathbb{R}$ be a smooth $G$-invariant
Morse-Bott function \cite{Bott54}. This means that the critical points of
$f$ form a union of disjoint connected submanifolds $Y_{1}, \ldots, Y_{r'}$
such that for every $x\in Y_{i}$ the Hessian of $f$ is
nondegenerate on all subspaces of $T_{x}M$ intersecting
$T_{x}Y_{i}$ transversally. One verifies directly that the index of the Hessian of $f$ is constant
on any orbit $G\cdot Y_{i}$. Set $\big\{B_{1},
\ldots, B_{r}\big\}=\big\{G\cdot Y_{1}, \ldots, G\cdot Y_{r'}\big\}$, $r\leqslant r'$, where $B_{1}, \ldots, B_{r}$
are pointwise disjoint orbits.
Then $B_{i}$ is a $G$-invariant submanifold of $M$. For $1\leqslant i\leqslant r$,
let $n_{i}$ be the dimension of the submanifold $B_{i}$ and $n_{i}^{-}$ be the index of
the Hessian of $f$ on $B_{i}$.

Using the equivariant Morse's Lemma \cite{Wasserman69}, we embed each critical submanifold
$B_{i}$ in a $G$-invariant tubular neighborhood $(h, N^{-}_{i}\oplus N_{i}^{+})$ of $B_{i}$
such that $h$ equivariantly embeds $N^{-}_{i}\oplus N^{+}_{i}$ into $M$.
Moreover, there is an open $G$-invariant neighborhood
$\mathcal {B}_{i}$ of $B_{i}$ in $N^{-}_{i}\oplus N_{i}^{+}$ such that if $Z=(Z^{-}, Z^{+})\in
\mathcal {B}_{i}$, then
\begin{align}\label{1.1} f\circ h(Z^{-},
Z^{+})=c-\frac{|Z^{-}|^2}{2}+\frac{|Z^{+}|^2}{2},
\end{align}
where $c$ denotes the value of the constant fundtion $f|_{B_{i}}$. The rank of $N^{-}_{i}$ is $n^{-}_{i}$, while that of
$N^{+}_{i}$ is $m-n_{i}-n^{-}_{i}$. Let $o(N_{i}^{-})$ denote the
orientation bundle of $N_{i}^{-}$. We call $n^{-}_{i}$ the index of
$B_{i}$ in $M$.

In the sequel, we will often omit the subscript $i$ in $B_{i},
n_{i}, n_{i}^{-}$, i.e., $n$ denotes the dimension of the critical
submanifold $B$ and $n^{-}$ is the index.
Denote by $o(N^{-})$ the orientation bundle of $N^{-}$ over $B$.

Let $W_{1}, W_{2}$ be two finite-dimensional $G$-representations.
A morphism between $W_{1}$ and $W_{2}$ is a linear map which is
$G$-equivariant. Let $\textrm{Hom}_{G}(W_{1}, W_{2})$ denote the set of
all morphism between $W_{1}$ and $W_{2}$. If $E_{1}, E_{2}$ are two
finite-dimensional representations of $G$, then we say that
\begin{align}\label{1.2}
E_{1}\leqslant E_{2}
\end{align} in the representation ring
$R(G)$ if for any irreducible representation $V$ of $G$, the
multiplicity of $V$ in $E_{1}$ is smaller than the multiplicity of
$V$ in $E_{2}$, equivalently,
\begin{align}\label{1.3}\textrm{dim}\
\textrm{Hom}_{G}(V,E_{1})\leqslant \textrm{dim}\
\textrm{Hom}_{G}(V,E_{2}).
\end{align}

Denote by $\Omega^{i}(B, o(N^{-}))$ the space of smooth differential
$i$-forms on $B$ with values in $o(N^{-})$. Set
$\Omega(B,o(N^{-}))=\bigoplus^{n}_{i=0}\Omega^{i}(B, o(N^{-}))$. Let $d^{B}$ denote
the exterior differential on $\Omega(B, o(N^{-}))$ induced by the flat connection $\nabla^{o(N^{-})}$
on $o(N^{-})$. Denote by $H^{\bullet}(B, o(N^{-}))$ the cohomology of
the de Rham complex $\big(\Omega(B, o(N^{-})), d^{B}\big)$.
Let $H^{\bullet}(M)$ denote the de Rham cohomology groups of $M$.

The main result of this paper is as follows.
\begin{thm} \label{t1.1}
Let $M$ be a smooth $m$-dimensional closed and connected
manifold, and let $G$ be a finite group acting smoothly on $M$.
Let $f: M\rightarrow \mathbb{R}$ be a smooth $G$-invariant
Morse-Bott function.
Then we have for $k=0,  1,  \ldots, m$,
\begin{align}\label{1.4}
\sum_{j=0}^{k}(-1)^{k-j}H^{j}(M)\leqslant
\sum_{i=1}^{r}\sum_{j=n_{i}^{-}}^{k}(-1)^{k-j}H^{j-n_{i}^{-}}(B_{i},
o(N^{-}_{i})).
\end{align}  \noindent in the sense of (\ref{1.2}). When
$k=m$, the equality holds,
\begin{align}\label{1.5}
\sum_{j=0}^{m}(-1)^{m-j}H^{j}(M) =
\sum_{i=1}^{r}\sum_{j=n_{i}^{-}}^{m}(-1)^{m-j}H^{j-n_{i}^{-}}(B_{i},
o(N^{-}_{i})).
\end{align} \end{thm}
\noindent
Let us explain Theorem \ref{t1.1} in more detail.
Set
\begin{align}\label{1.6}
F_{j}=\bigoplus_{i=1}^{r} H^{j-n_{i}^{-}}(B_{i},
o(N^{-}_{i})), \ \ q_{j}=\sum_{i=1}^{r}\textrm{dim}\ H^{j-n_{i}^{-}}(B_{i},
o(N^{-}_{i})).
\end{align}
Let $\{V^{\alpha}\}^{l_{0}}_{\alpha=1}$ be the finite set of irreducible representations
of $G$.
As representation spaces of $G$, $F_{j}$ and $H^{j}(M)$ have the
following decompositions:
\begin{align}\label{1.7}
F_{j}=\bigoplus^{l_0}_{\alpha=1} \textrm{Hom}_{G}(V^{\alpha},
F_{j})\otimes V^{\alpha},\ \ H^{j}(M)=\bigoplus^{l_{0}}_{\alpha=1} \textrm{Hom}_{G}(V^{\alpha},
H^{j}(M))\otimes V^{\alpha}.
\end{align}
For $k=0, 1, \ldots, m, \alpha=1,
\ldots, l_{0}$, set
\begin{align}\label{1.8}
d_{j}^{\alpha}=\textrm{dim} \ \textrm{Hom}_{G}(V^{\alpha}, F_{j}), \ \
b^{\alpha}_{j}=\textrm{dim} \
\textrm{Hom}_{G}(V^{\alpha}, H^{j}(M))
\end{align}
Then (\ref{1.4}) is equivalent to
\begin{align}\label{1.9}\sum_{j=0}^{k}(-1)^{k-j}b_{j}^{\alpha}\leqslant
\sum_{j=0}^{k}(-1)^{k-j}d_{j}^{\alpha}, \end{align} and (\ref{1.5})
is equivalent to:
\begin{align}\label{1.10}\sum_{j=0}^{m}(-1)^{m-j}b_{j}^{\alpha}=
\sum_{j=0}^{m}(-1)^{m-j}d_{j}^{\alpha}.
\end{align}

From the equivariant Morse inequalities (\ref{1.9}) and
(\ref{1.10}), we will obtain the Morse inequalities for manifolds with
nonempty boundary. This goes like follows.
Let $M$ be a smooth $m$-dimensional connected orientable manifold with nonempty
boundary $\partial M$. Let $f: M\rightarrow \mathbb{R}$ be a smooth
function which is a Morse-Bott function in the interior of
$M$. Let $f|_{\partial M}$ be restriction of $f$ to the boundary.
We also assume the following condition. Let
$\partial M=N_{+}\sqcup N_{-}$ be a disjoint union of closed
manifolds such that $f(u, y)=\frac{1}{2}u^{2}+f_{+}(y)$ in a collar
neighborhood $N_{+}\times [0, \eta)$ of $N_{+}$, while $f(u,
y)=-\frac{1}{2}u^{2}+f_{-}(y)$ in a collar neighborhood $N_{-}\times [0,
\eta)$ of $N_{-}$, where $f_{+}$ (resp. $f_{-}$) is a Morse-Bott
function on $N_{+}$ (resp. $N_{-}$). This implies that $f|_{\partial M}$ is also
a Morse-Bott function.

Let $N_{+}=N_{a+}\sqcup N_{r+}$ and $N_{-}=N_{a-}\sqcup N_{r-}$ be
disjoint union of closed manifolds. The subscripts $"a"$ and $"r"$
refer to absolute and relative boundary conditions, respectively. Set
$N_{a}=N_{a+}\sqcup N_{a-}, N_{r}=N_{r+}\sqcup N_{r-}$. The Riemannian
metric is assumed to take the
product form $g^{TM}=g^{T\partial M}\oplus
d^{2}u$ in the collar neighborhood $\partial M\times [0, \eta)$,
where $g^{T\partial M}$ is a Riemannian metric on
$\partial M$.

Let $\{B_{i}\}^{r}_{i=1}$ (resp. $\{S_{+, i}\}^{t_{+}}_{i=1}$, resp.
$\{S_{-, i}\}^{t_{-}}_{i=1}$) be the critical submanifolds of $f$ in
the interior of $M$ (resp. of $f_{+}$ on $N_{+}$, resp. of $f_{-}$
on $N_{-}$). Set
\begin{align} \label{1.11}S_{a+, i}=S_{+, i}\cap N_{a}, \ \
S_{r-, i}=S_{-, i}\cap N_{r}.
\end{align}
Let $o(N^{-}_{i})$ denote the orientation bundle of $N^{-}_{i}$ over
$B_{i}$ as before. To simplify our notation, we denote by $o(S_{a+, i})
$ (resp. $o(S_{r-, i})$) the corresponding bundle over $S_{a+, i}$
(resp. $o(S_{r-, i})$) and by $n^{-}_{a+,i}$ (resp. $n^{-}_{r-,i}$)
its index in $N_{a}$ (resp. $N_{r}$).
Set
\begin{align}\label{1.12}
F_{a+,j}=&\bigoplus^{t_{+}}_{i=1}H^{j-n^{-}_{a+,i}}\big(S_{a+, i},
o(S_{a+,i})\big), \ \ q_{a+,j}=\textrm{dim}\  F_{a+,j};
\nonumber\\F_{r-,j}=&\bigoplus^{t_{-}}_{i=1}H^{j-n^{-}_{r-,i}}\big(S_{r-,
i}, o(S_{r-,i})\big), \ \ q_{r-,j}=\textrm{dim}\  F_{r-,j}.
\end{align}
Denote by $H^{\bullet}(M, N_{r})$ the relative cohomology of $M$
with respect to $N_{r}$.

\begin{thm} \label{t1.2}  The following inequalities hold for $k=0,
1, \ldots, m$,
\begin{align}\label{1.13}\sum_{j=0}^{k}(-1)^{k-j}\beta_{j}(M, N_{r})\leqslant
\sum_{j=0}^{k}(-1)^{k-j}\mu_{j}, \end{align} where
\begin{align} \label{1.14}
\beta_{j}(M, N_{r})=\textup{dim}\ H^{j}(M, N_{r}),\ \
\mu_{j}=q_{j}+q_{a+,j}+q_{r-, j-1}\,.
\end{align}  The
equality holds for $k=m$. \end{thm}
When $f|_{\partial M}=0$ and the critical points of $f$ in the interior of $M$
are isolated and nondegenerate, Theorem \ref{t1.2} reduces to Theorem 1 in
\cite{Zadeh10}.

\section{Equivariant Morse Inequalities}

This section is organized as follows. In Section \ref{s0.1}, we calculate the kernel of
the Witten Laplacian on Euclidean space. The results of this section will
be applied to the fibres of the normal bundle of critical manifolds in $M$.
In Section \ref{s0.2}, a special metric on the total space $N$ is constructed such that the
critical manifolds are totally geodesic in $N$.
In Section \ref{s0.3}, we introduce the Witten deformation, the deformed de Rham operator
$D_{T}$ and state a crucial result (Proposition \ref{t2.2}) concerning the lower part of the spectrum of $D_{T}$.
Section \ref{s0.4} is devoted to the Taylor expansion of $D_{T}$ near the
critical manifolds.
In Section \ref{s0.5}, a decomposition of $D_{T}$ is established.
Various estimates are also briefly described there.
Finally in Section \ref{s0.6}, we prove Proposition \ref{t2.2}
and then finish the proof of Theorem \ref{t1.1}.

\subsection{Some calculations on Euclidian space}\label{s0.1}

In this section, we calculate the kernel of
the Witten Laplacian on Euclidian space.
The result of this section will be applied to the fibres of the
normal bundle to $B$ in $M.$

Let $V$ be an $l$-dimensional real vector space endowed with an
Euclidean scalar product.  Let $V^{+}, V^{-}$ be two subspaces such
that $V=V^{-} \oplus V^{+}$ and $\textrm{dim} V^{-}=n^{-}$.
Let $f\in C^{\infty}(V, \mathbb{R})$ be defined as:
\begin{align}\label{2.3}
f(Z)=f(0)-\frac{{\left\vert Z^{-}\right\vert}^2}{2}+
\frac{{\left\vert Z^{+}\right\vert}^2}{2},
\end{align} where
$Z^{-}=(Z_{1}, \ldots, Z_{n^{-}}), Z^{+}=(Z_{n^{-}+1}, \ldots,
Z_{l}), (Z^{-}, Z^{+})$ denote the coordinate functions on $V$
corresponding to the decomposition $V=V^{-}\oplus V^{+}$.

Let $Z=\sum_{\alpha=1}^{l}Z_{\alpha}e_{\alpha}$
 be the radial vector field on $V$. There is a
natural Euclidean scalar product on $\Lambda V^{\ast}$. Let
$dv_{V}(Z)$ be the volume form on $V$. Let $S$ be the set of
the square integrable sections of $\Lambda V^{\ast}$ over $V$. For
$s_{1}, s_{2} \in S$, set
\begin{align}\label{2.4}
\big\langle s_{1}, s_{2}\big\rangle=
\int_{V}\langle s_{1}, s_{2}\rangle_{\Lambda V^{\ast}}dv_{V}(Z).
\end{align} Let $d$ be the exterior differential operator acting on the
smooth section of $\Lambda V^{\ast}$, and let $\delta$ be the formal
adjoint of $d$ with respect to the Euclidean product (\ref{2.4}).

Let $C(V)$ be the Clifford algebra of $V$, i.e., the algebra
generated over $\mathbb R$ by $e\in V$ and the commutation
relations $e e'+e'e=-2\langle e, e'\rangle$ for $e, e'\in V$.
Let $c(e), \widehat{c}(e)$ be the Clifford
operators acting on $\Lambda V^{\ast}$ defined by
\begin{align} \label{2.5}c(e)=e^{\ast}\wedge-i_{e}, \
\widehat{c}(e)=e^{\ast}\wedge+i_{e}, \end{align} where
$e^{\ast}\wedge$ and $i_{e}$ are the standard notation for exterior
and interior multiplication and $e^{\ast}$ denotes the dual of $e$
with respect to the Euclidean scalar product on $V$. Then $\Lambda V^{\ast}$ is
a Clifford module.
If $X, Y\in V,$ one has
\begin{align}\label{2.6}
c(X)c(Y)+c(Y)c(X)&=-2\langle X, Y\rangle, \nonumber\\
\widehat{c}(X)\widehat{c}(Y)+\widehat{c}(Y)\widehat{c}(X)&=2\langle X, Y\rangle,\\
c(X)\widehat{c}(Y)+\widehat{c}(Y)c(X)&=0.  \nonumber \end{align}
If we denote by $v$ the gradient of $f$ with respect to the given
Euclidean scalar product, then \begin{align}\label{2.7}
v(Z)=-\sum_{\alpha=1}^{n^{-}}Z_{\alpha}e_{\alpha}
+\sum_{\alpha=n^{-}+1}^{l}Z_{\alpha}e_{\alpha}.
\end{align}
Let $\Delta$ be the standard Laplacian on $V$, i.e.,
\begin{align}\label{2.8}
\Delta=-\sum^{l}_{\alpha=1}\Big(\frac{\partial}{\partial
Z_{\alpha}}\Big)^{2}. \end{align}
Set
\begin{align}
d_{T}=e^{-Tf}d\cdot e^{Tf},\ \delta_{T}=e^{Tf}\delta\cdot e^{-Tf}.
\end{align}
The deformed de Rham operator on the Euclidean space is defined by
\begin{align}
D_{T, v}=d_{T}+\delta_{T}=d+\delta+T\hat{c}(v).
\end{align}
Let $e^{1}, \ldots, e^{l}$ be the dual basis of $e_{1}, \ldots, e_{l}$.
Then we have the following result \cite{Witten82}, \cite[Prop.\,4.9]{Zhang01}.

\begin{prop}\label{t2.3}  The kernel of $D_{T, v}^2$ is one-dimensional
and is spanned by
\begin{align}\label{2.9}
\beta=\exp\!\big(-\tfrac{T{|Z|}^{2}}{2}\big)\,e^{1} \wedge\ldots\wedge
e^{n^{-}}.\end{align} Moreover, all nonzero eigenvalues of $D_{T, v}^{2}$
are $\geqslant 2T$. \end{prop}

\begin{proof} We recall the proof for the reader's convenience.
For $e\in V$, let $\nabla_{e}$ be the differential operator
along the vector $e$. It is easy to calculate the square of $D_{T, v}$,
\begin{align}\begin{split}\label{2.10}
D_{T, v}^{2}&=\Delta+T^2{|Z|}^{2}+
T\sum_{\alpha=1}^{l}c(e_{\alpha})\widehat{c}(\nabla_{e_{\alpha}}v)
\\&=(\Delta+T^2{|Z|}^{2}-Tl)+T\sum_{\alpha=1}^{n^{-}}\big[1-c(e_{\alpha})
\widehat{c}(e_{\alpha})\big]
+T\sum_{\alpha=n^{-}+1}^{l}\big[1+c(e_{\alpha})\widehat{c}(e_{\alpha})\big]
\\&=(\Delta+T^2{|Z|}^{2}-Tl)+2T\Big(\sum_{\alpha=1}^{n^{-}}i_{e_{\alpha}}
{e^{\alpha}}\wedge +\sum_{\alpha=n^{-}+1}^{l}{e^{\alpha}}\wedge
i_{e_{\alpha}}\Big).\end{split}\end{align}
The operator
\begin{align}\label{2.11}\mathcal{L}_{T}=\Delta+T^2{|Z|}^{2}-Tl
\end{align} is the harmonic oscillator operator on $V$.
By \cite[Th.\,1.5.1]{Glimm87}, \cite[Appendix E]{Ma07}, we know
that $\mathcal{L}_{T}$ is a positive elliptic operator with one-dimensional
kernel generated by $\exp(-\tfrac{T{|Z|}^{2}}{2})$. 
Moreover, the nonzero eigenvalues of $\mathcal{L}_{T}$ are all
greater than $2T$.  It is also easy to verify that the linear
operator
\begin{align}\label{2.12}
\sum_{\alpha=1}^{n^{-}}i_{e_{\alpha}}{e^{\alpha}}\wedge
+\sum_{\alpha=n^{-}+1}^{l}{e^{\alpha}}\wedge
i_{e_{\alpha}}\end{align} is positive and has one-dimensional kernel generated by
$e^{1}\wedge\ldots\wedge e^{n^{-}}$.
The proof of Proposition \ref{t2.3} is complete. \end{proof}

\subsection{Local analysis near critical manifolds}\label{s0.2}

Let $B$ be an equivariant critical submanifold of the Morse-Bott function $f$.
By equivariant Morse's Lemma \cite[Lemma 4.1]{Wasserman69}, we know that $B$
possesses a $G$-invariant tubular neighborhood $(h, N)$ such that:
\\(1) $N$ is a $G$-vector bundle over $B$, which is endowed with $G$-invariant
scalar product $g^{N}$. Moreover $N$,  which has rank $m-n$, splits
into two orthogonal $G$-subbundles $N=N^{-}\oplus N^{+}$, where the rank
of $N^{-}$ is $n^{-}$.
\\(2) $h$ equivariantly embeds $N$ into $M$. Moreover there is
an open $G$-invariant neighborhood
$\mathcal {B}$ of $B$ in $N$ such that if $Z=(Z^{-}, Z^{+})\in
\mathcal {B}$, then
\begin{align}\label{2.14}
f(h(Z))=c-\frac{{|Z^{-}|}^2}{2}+\frac{{|Z^{+}|}^2}{2},
\end{align} where $c$ denotes the value of the constant function $f|_B$.

In the sequel, we will identify $N$ and $h(N)$. Let $\pi$ be the
projection $N \rightarrow B.$ We denote by $o(N^{-})$
the orientation bundle of $N^{-}$.

Let $g^{TB}$ be a $G$-invariant Riemannian metric on $TB$ and $\nabla^{TB}$ be the
Levi-Civita connection on $TB$ which is then $G$-invariant. As Euclidean
$G$-bundles, $(N^{-}, g^{N^{-}})$ (resp. $(N^{+},
g^{N^{+}}))$ can be endowed with $G$-invariant Euclidean connections
$\nabla^{N^{-}}$ (resp. $\nabla^{N+}$), where $g^{N^{-}}$ (resp. $g^{N^{+}}$) denotes
the restriction of the scalar product $g^{N}$ to the subbundle $N^{-}$ (resp. $N^{+}$).
We then have a natural Euclidean connection $\nabla^{N}$ on $N$, i.e.,
\begin{align}\label{2.15}
\nabla^{N}=\nabla^{N^{-}}\oplus
\nabla^{N^{+}}.
\end{align} The Euclidean connection $\nabla^{N}$ on
$N$ induce a splitting $TN=T^{H}N \oplus T^{V}N$ of the tangent
space of the total space $N$ \cite[Prop.\,1.20]{Berline04}, where $T^{H}N$ is the
horizontal part of $TN$ with respect to the Euclidean connection
$\nabla^{N}$. If $X\in TB$, let $X^{H}$ denote the horizontal lift
of $X$ in $T^{H}N$ such that $X^{H}\in T^{H}N, {\pi}_{\ast}X^{H}=X.$

If $y \ \in N$, then $\pi_{\ast}$ identifies $T^{H}_{y}N$ with
$T_{\pi(y)}B$. Moreover,  $T^{V}_{y}N$ and $N$ can be naturally
identified. In this way, $T^{H}_{y}N$ and $T^{V}_{y}N$ are both
endowed with a scalar product. We can assume as well that they are
orthogonal for the metric $g^{TN}$ which splits into
$g^{TN}=\pi^{\ast}(g^{TB})\oplus g^{N}$.
Let $\nabla^{TN}$ be the Levi-Civita connection on
$N$ associated to the Riemannian metric $g^{TN}$.

Let $TN|_{B}$ be the restriction of the tangent bundle $TN$ to $B$.
Recall that $N$ is identified with the bundle orthogonal to $TB$ in
$TN|_{B}$, i.e., $TN|_{B}=TB\oplus N$. Let $\nabla^{TN|_{B}}$ be the
restriction of $\nabla^{TN}$ to $TN|_{B}$.

\begin{lemma}\label{t2.6} The following identity holds:
 \begin{align}
    \label{2.16} \nabla^{TN|_{B}}=\nabla^{TB}\oplus \nabla^{N}.
 \end{align}
\end{lemma}

\begin{proof}
The proof is straightforward and is left to the reader.
\end{proof}

\subsection {The deformed de Rham operator}\label{s0.3}

Let $g^{TM}$ be a $G$-invariant Riemannian metric on $M$ which coincides with $g^{TN}$ in a
neighborhood of $B$ via the embedding $h$ (this
is always possible by a partition of unity argument).

Let $o(TM)$ be the orientation line bundle on $M$ and let $dv_{M}$
be the density (or Riemannian volume form) on $M$. Note that we do
not assume that $M$ is oriented; thus $dv_{M}\in C^{\infty}(M,
\Lambda^{m}(T^{\ast}M)\otimes o(TM))$ (see \cite[p. 29]{Berline04},
\cite[p. 88]{Bott822}).  Let $\textup{E}$
be the set of smooth sections of $\Lambda (T^{\ast}M)$ on $M$.
For $s_{1}, s_{2}\in \textup{E}$,
set \begin{align} \label{2.24}
\big\langle s_{1}, s_{2}\big\rangle=\int_{M}\langle s_{1},
s_{2}\rangle(x)dv_{M}(x).
\end{align}

Let $D^{M}$ be the classical Dirac operator on $M$,
i.e., $D^{M}=d+\delta$, where $d$ is the exterior differential operator
and $\delta$ is the adjoint of $d$ with respect to the metric
(\ref{2.24}).

Set \begin{align}
d_{T}=e^{-Tf}d\cdot e^{Tf}, \ \delta_{T}=e^{Tf}\delta\cdot e^{-Tf}.
\end{align}
The deformed de Rham operator $D_{T}$ is defined by
\begin{align}
D_{T}=d_{T}+\delta_{T}=D^{M}+T\hat{c}(\nabla f),
\end{align}
where $\nabla f$ is the gradient vector field of $f$ with respect to
the Riemannian metric $g^{TM}$ of $M$. We denote by $\Omega^{j}(M)$ the
smooth sections of $j$-forms of $M$. The next result describes the lower part of
the spectrum of $D^{2}_{T}$ for large $T$. It will be proved in Section \ref{s0.6}.

\begin{prop} \label{t2.2}
There exist $C_{0}>0, T_{0}>0$ such
that for $T>T_{0},$ the number of eigenvalues of
$D_{T}^2|_{\Omega^{j}(M)}$ in $[0, C_{0})$ equals $q_{j}$. Moreover, the direct sum of eigenspaces of
$D_{T}^{2}$ with eigenvalues in $[0, C_{0})$ is a $G$-vector space.
\end{prop}

\subsection {Local expansion of the operator
$D_{T}$ near the critical submanifold $B$}\label{s0.4}

We first introduce a coordinate system on $M$ near $B$. If $y \in B,
Z\in N_{y}$, let $y_{t}=\textrm{exp}_{y}(tZ), t\in {\mathbb R}$ be
the geodesic in $M$ with $y_{0}=y, \dot{y}_{0}=Z$, where $\dot{y}_{0}$
denotes $dy_{t}/{dt}$ evaluated at $t=0$. For
$\varepsilon >0,$ set ${\mathcal B}_{\varepsilon}={\{(y, Z)\in
N; y\in B, |Z| < \varepsilon}\}.$ In the following we denote $|Z|_{g^{N}_{y}}$ simply by $|Z|$. Since $B$ and $M$ are compact,
there exists $\varepsilon_{0} >0$ such that for $0<\varepsilon <
\varepsilon_{0}$, the map $(y,Z)\in N\rightarrow \textup{exp}_{y}Z\in M$
is a diffeomorphism from $ {\mathcal
B}_{\varepsilon}$ onto a tubular neighborhood ${\mathcal
U}_{\varepsilon}$ of $B$ in $M$. From now on, we identify ${\mathcal
B}_{\varepsilon}$ with ${\mathcal U}_{\varepsilon}$ and use the
notation $x=(y, Z)$ instead of $x=\textrm{exp}_{y}Z$. Finally, we
identify $y \in B$ with $(y, 0)\in N$.

The symbols $dv_{B}$ and $dv_{N}$ are
understood in the same manner as $dv_{M}$. Let
$\big\{f_{1}, \ldots, f_{n}$, $e_{1}, \ldots, e_{l}\big\}$
be a local orthonormal frame of $TN|_{B}$ with $\{f_{1}, \ldots, f_{n}\}$ being an
orthonormal frame for $TB$ and $\{e_{1}, \ldots, e_{l}\}$ an orthonormal frame for $N$.
By the definition of $g^{TN}$, we know that $e_{1}, \ldots, e_{l}$ are
also orthonormal basis at the points $(y, Z)$ on the total space
$N$. It is clear that
\begin{align}\label{2.23} dv_{N}(y, Z)=
dv_{B}(y)dv_{N_{y}}(Z). \end{align}

Take $\alpha>0$. Let ${\bf E}$ (resp. ${\bf
E}_{\alpha}$) be the set of smooth sections of $\pi^{\ast}(\Lambda T^{\ast}M|_{B})$ on the total
space of $N$ (resp. of $\pi^{\ast}(\Lambda (T^{\ast}M)|_{B})$ over
$\mathcal{B}_{\alpha}$).

For $s_{1}, s_{2}\in  {\bf E}$ have compact support,
set
\begin{align}\label{2.25}
\big\langle s_{1}, s_{2}\big\rangle=\int_{B}\Big(\int_{N_{y}}\langle s_{1}, s_{2}\rangle(y,
Z)dv_{N_{y}}(Z)\Big)dv_{B}(y). \end{align}

If $s\in {\bf E}$ has compact support in ${\mathcal
B}_{\varepsilon_{0}}$, we will identify $s$ with an element of
$\textrm{E}$ which has compact support in ${\mathcal
U}_{\varepsilon_{0}}$. This identification is unitary with respect
to the Euclidean product (\ref{2.24}) and (\ref{2.25}).

The Levi-Civita connection $\nabla^{TM}$ on $TM$ induces a connection
on $\Lambda (T^{\ast}M)$, which we denote by
$\nabla^{\Lambda (T^{\ast}M)}$.  Let $\nabla^{\Lambda (T^{\ast}M)|_{B}}$
be the restriction of $\nabla^{\Lambda (T^{\ast}M)}$ to $\Lambda
(T^{\ast}M)|_{B}$. The connection $\nabla^{\Lambda (T^{\ast}M)|_{B}}$ on
$\Lambda (T^{\ast}M)|_{B}$ can be lift to a connection on the bundle
$\pi^{\ast}(\Lambda (T^{\ast}M)|_{B})$, which we denote by
$\pi^{\ast}(\nabla^{\Lambda (T^{\ast}M)|_{B}}).$

\begin{defn}\label{t2.7} Let $D^{H}, D^{N}$ be the following operators acting on ${\bf E}$:
  \begin{align}\begin{split}\label{2.26}
  D^{H}&=\sum_{j=1}^{n}c(f_{j})\pi^{\ast}(\nabla^{\Lambda (T^{\ast}M)|_{B}})_{f_{j}^{H}},
 \\ D^{N}&=\sum_{\alpha=1}^{l}c(e_{\alpha})
 \pi^{\ast}(\nabla^{\Lambda (T^{\ast}M)|_{B}})_{e_{\alpha}}.
 \end{split}\end{align}
\end{defn}

One verifies directly that $D^{H}, D^{N}$ is self-adjoint with respect to
metric (\ref{2.25}). Indeed, $D^{N}$ is formally self-adjoint along the
fibres of $N$, i.e., for $s_{1}, s_{2} \in \textrm{\bf E}$ with
compact supports, $y\in B$,
\begin{align}\label{2.27}\int_{N_{y}}\langle D^{N}s_{1},
s_{2}\rangle(y, Z)dv_{N_{y}}(Z)=\int_{N_{y}}\langle s_{1},
D^{N}s_{2}\rangle(y, Z)dv_{N_{y}}(Z). \end{align}

Using the identification $(\Lambda (T^{\ast}M))_{(y, Z)}$ with $(\Lambda
(T^{\ast}M))_{y}$ by parallel transport along the geodesic
$t\rightarrow (y, tZ), t \in [0, 1]$ with respect to the connection
$\nabla^{\Lambda (T^{\ast}M)}$, we can now consider the connection
$\nabla^{\Lambda (T^{\ast}M)}$ as a Euclidean connection on
$\pi^{\ast}(\Lambda (T^{\ast}M)|_{B})$ over ${\mathcal
B}_{\varepsilon}$.

Recall that the vector field
$v$ is defined as in (\ref{2.7}). Set
\begin{align}\label{2.30}
D_{T}^{N}=D^{N}+T\widehat{c}(v).
\end{align} Then we have the
following  analogue of \cite[Lemma 2.4]{Feng01} ,
\cite[Th.\,8.18]{Bismut74}.

\begin{thm}\label{t2.8}
The following
asymptotic formula holds on ${\bf E}_{\varepsilon_{0}}$ as $T\rightarrow +\infty$,
\begin{align}\label{2.31}
D_{T} =D_{T}^{N}+D^{H}+O\big({|Z|}^2 {\partial}^{N}+|Z|^2
\partial^{H}+|Z|+T{|Z|}^4\big), \end{align}
where $\partial^{H}$ and $\partial^{N}$ represent horizontal and
vertical differential operators, respectively.
 \end{thm}

\begin{proof} We adapt the proof from \cite[Th.\,8.18]{Bismut74},
\cite[Lemma 2.3]{Feng01} and show how the proof simplifies in our case,
due to the fact that $B$
is now a totally geodesic submanifold of the total manifold $N$.
For $({y}, Z)\in \mathcal{B}_{\varepsilon_{0}}, X\in T_{{y}}N$, let
$\tilde{X}$ be the parallel transport of $X$ with respect to the
connection $\nabla^{TM}$ along the geodesic $t\rightarrow ({y}, tZ), \
t\in [0, 1]$, i.e.,
\begin{align}\label{2.32}
(\nabla_{Z}^{TM}\tilde{X})({y}, Z)=0.
\end{align}
\noindent  Then $\tilde{e}_{\alpha}({y}, Z)=e_{\alpha}({y})$ and
\begin{align}\label{2.33}
D^{M}=\sum_{j=1}^{n}c(f_{j})\nabla^{\Lambda (T^{\ast}M)}_{\tilde{f}_{j}}
+\sum^{l}_{\alpha=1}c(e_{\alpha})\nabla^{\Lambda (T^{\ast}M)}_{e_{\alpha}}.
\end{align}
For $1\leqslant j\leqslant n$,  set
\begin{align}\label{2.34}\tilde{f}_{j}({y},
Z)=f_{j}({y})+\sum_{k=1}^{n}\sum_{\alpha=1}^{l}c_{kj}^{\alpha}({y})Z_{\alpha}f_{k}
+\sum_{\beta=1}^{l}\sum_{\alpha=1}^{l}c_{\beta
j}^{\alpha}({y})Z_{\alpha}e_{\beta}+O(|Z|^2),\end{align}
where $c_{kj}^{\alpha}({y}),\ c_{\beta j}^{\alpha}({y})$ are smooth functions of $y$.
Using (\ref{2.32}) and Lemma \ref{t2.6}, we find that
\begin{align}\label{2.35}
\tilde{f}_{j}({y},Z)=f^{H}_{j}({y},Z)+O(|Z|^{2}).
\end{align}
Set \begin{align} \label{2.36}
\Gamma=\nabla^{\Lambda (T^{\ast}M)}-\pi^{\ast}(\nabla^{\Lambda (T^{\ast}M)|_{B}}).
\end{align}
By Lemma \ref{t2.6}, $\Gamma_{{y}}=0.$
Combining (\ref{2.33}), (\ref{2.35}) and (\ref{2.36}),  we get
\begin{align}\label{2.37}
D^{M}=D^{H}+D^{N}+O\big(|Z|^{2}\partial^{H}+|Z|^{2}\partial^{N}+|Z|\big).
\end{align}
Set \begin{align}\label{2.38}
\nabla f({y}, Z)=\sum_{j=1}^{n}v_{j}({y},
Z)\tilde{f}_{j}+\sum_{\alpha=1}^{l}v_{\alpha}({y}, Z)e_{\alpha},
\end{align} where \begin{align}\label{2.39}v_{j}({y},
Z)=(\tilde{f}_{j}f)({y}, Z),\ \ v_{\alpha}({y}, Z)=(e_{\alpha}f)({y}, Z).
\end{align}
Using (\ref{2.14}), we find that
\begin{align}\label{2.40}
v_{\alpha}({y}, Z)=-Z_{\alpha}, \ \textup{if}\ 1\leqslant \alpha\leqslant n^{-};
\ \ v_{\alpha}({y}, Z)=-Z_{\alpha}, \ \textup{if}\ n^{-}+1\leqslant \alpha \leqslant l.
\end{align}
From  (\ref{2.14}) and ({\ref{2.35}}), we have
\begin{align}\label{2.41}
v_{j}({y}, Z)=O(|Z|^{4}).
\end{align}
Substituting (\ref{2.40}) and (\ref{2.41}) into (\ref{2.38}), we get
\begin{align}\label{2.42}
\nabla f({y},Z)=v+O(|Z|^{4}).
\end{align}
Now (\ref{2.31}) follow immediately from (\ref{2.37}) and (\ref{2.42}).
 \end{proof}

Note that $D^{N}_{T}$ is actually an elliptic operator acting
fibrewise on $\pi^{\ast}(\Lambda N^{\ast})$. We now formalize Witten's description of
the spectrum of $D^{2}_{T}$ (\cite[pp. 674-675]{Witten82}) for the equivariant case by using the
argument of \cite[Prop.\,4.9]{Zhang01}.

\begin{thm}\label{t2.9}
For any $y\in B$, the restriction of $(D_{T}^{N})^2$
to $C^{\infty}(N_{y}, \Lambda N^{\ast}_{y})$ is a positive operator
with kernel generated by
\begin{align}\label{2.43}
\beta_{y}=\exp\!\big(-\tfrac{T{|Z|}^{2}}{2}\big)\,\theta_{y},\end{align}
where $\theta_{y}$ is the volume form of $N^{-}_{y}$. Moreover, all
the nonzero eigenvalues of $(D_{T}^{N})^{2}$ on
$C^{\infty}(N_{y}, \Lambda N^{\ast}_{y})$ are  $\geqslant 2T$.
\end{thm}

\begin{proof} Let $\Delta^{N}$ be the positive Laplacian along the fibres of
$N$.
From (\ref{2.30}), it is clear that on
$\pi^{\ast}\big(\Lambda (T^{\ast}M)|_{B}\big)=\pi^{\ast}\big(\Lambda T^{\ast}B\big)\otimes \Lambda N^{\ast}$,
\begin{align}\label{2.44}
(D_{T}^{N})^2=-\sum_{\alpha=1}^{l}
(\pi^{\ast}\nabla^{TM|_{B}}_{e_{\alpha}})^{2}+T^{2}|v|^{2}
+T\sum_{\alpha=1}^{l}c(e_{\alpha})
\widehat{c}(\pi^{\ast}\nabla^{TM|_{B}}_{e_{\alpha}}v).
\end{align}
By (\ref{2.7}), we obtain that
\begin{align}\label{2.45}(D_{T}^{N})^2=\Delta^{N}+T^{2}|Z|^{2}
-T\sum_{\alpha=1}^{n^{-}}c(e_{\alpha})\widehat{c}(e_{\alpha})
+T\sum_{\alpha=n^{-}+1}^{l}c(e_{\alpha})\widehat{c}(e_{\alpha}).\end{align}
Hence Theorem \ref{t2.9} follows
from Proposition \ref{t2.3}. \end{proof}

\subsection {Estimates of the components of $D_{T}$ as $T\rightarrow
+\infty$.}\label{s0.5}

In this section, we will give a decomposition of $D_{T}=
\sum_{j=1}^{4}D_{T, j}$ (see (\ref{2.53})) and establish estimates of $D_{T,
j}$ as $T\rightarrow +\infty$ by using Bismut-Lebeau analytic
localization techniques \cite{Bismut74}.

We denote by
$\textup{det}(N^{-})^{\ast}$ the determinant line bundle of
$(N^{-})^{\ast}$.
The connection $\nabla^{N^{-}}$ on $N^{-}$ induces naturally an Euclidean connection
$\nabla^{\textup{det}(N^{-})^{\ast}}$ on  $\textup{det}(N^{-})^{\ast}$.
Let $\Phi: \textup{det}(N^{-})^{\ast}
\rightarrow o(N^{-})$ denote the canonical isomorphism
over $B$. Let $\nabla^{o(N^{-})}$ be the Euclidean connection on $o(N^{-})$ induced by
$\nabla^{\textup{det}(N^{-})^{\ast}}$ via canonical isomorphism
$\Phi: \textup{det}(N^{-})^{\ast}\rightarrow o(N^{-})$.

For any $\mu>0$, let $\textrm{E}^{\mu}$ (resp. $\textrm{\bf
E}^{\mu}$, resp. $\textrm{F}^{\mu}$) be the set of sections of
$\Lambda (T^{\ast}M)$ on $M$ (resp. of
$\pi^{\ast}\big(\Lambda(T^{\ast}M)|_{B}\big)$ on the total space $N$, resp. of
$\Lambda (T^{\ast}B)\otimes o(N^{-})$ on $B$) which lies in the $\mu$-th Sobolev
spaces. Let $\big\|\cdot\big\|_{\textrm{E}^{\mu}}$ (resp. $\big\|\cdot
\big\|_{\textrm{\bf E}^{\mu}}$, resp. $\big\|\cdot
\big\|_{\textrm{F}^{\mu}}$) be the Sobolev norm on
$\textrm{E}^{\mu}$(resp. $\textrm{\bf E}^{\mu}$, resp.
$\textrm{F}^{\mu}$). We will always assume that the norm
$\big\|\cdot\big\|_{\textrm{E}^{0}}$ (resp. $\big\|\cdot\big\|_{{\bf E}^{0}}$)
is the norm associated with the Euclidean product (\ref{2.24}) (resp.
(\ref{2.25})). The norm $\big\|\cdot\big\|_{\textrm{F}^{0}}$ defined
on the sections of $\Lambda (T^{\ast}B)\otimes o(N^{-})$ is associated with a Euclidean
product similarly to (\ref{2.24}).

Take $\varepsilon\in (0, \frac{\varepsilon_{0}}{2}]$. Let $\varphi$
be a smooth function on $\mathbb{R}$ with values in [0, 1] such that
\begin{align}\label{2.46}
 \varphi(a) = \left\{%
 \begin{array}{ll}
 1 &\textrm{if }a\leqslant\frac{1}{2},\\
 0 &\textrm{if } a\geqslant 1.
 \end{array}
 \right.\end{align}
For $y\in B,Z\in N_{y}$, set \begin{align}\label{2.47}\rho({
Z})=\varphi\big(\tfrac{|Z|}{\varepsilon}\big). \end{align} For $T
>0$, set
\begin{align}\label{2.48}
\alpha_{T}(y)=\int_{{ N}_{y}}\textrm{exp}(-T {\left\vert {
Z}\right\vert}^2)\rho^2({ Z})dv_{N_{y}}({\textrm Z}).
\end{align}
Clearly, $y\mapsto \alpha_{T}(y)$ is a constant function on $B$.
Since for $|Z| \leqslant
\varepsilon/2, \ \rho({ Z})=1$, there exist $c>0, C>0$ such that for
$T \geqslant 1$, \begin{align} \label{2.49}
\frac{c}{T ^{l/2}}
\leqslant \alpha_{T } \leqslant \frac{C}{T ^{l/2}}. \end{align} Here
$l=m-n$ denotes the rank of $N$.

\begin{defn}\label{t2.10}
   For $\mu\geqslant 0, T
   >0$, define $J_{T}: \textrm{F}^{\mu}\rightarrow {\bf E}^{\mu}$ by
  \begin{align}\label{2.50}
  J_{T}s(y, Z)=\frac{1}{\sqrt{\alpha_{T}}}\, \rho({Z})\exp\!\big(\!-\tfrac{T
  |Z|^{2}}{2}\big) s(y)\wedge \theta_{y} \in {\bf E}^{\mu},\ \ s\in \textrm{F}^{\mu},
  \end{align}
where the smooth section
$\theta$ of $\Lambda ^{n^{-}}(N^{-})^{\ast}\otimes o(N^{-})$
is given by \begin{align}
u^{1}\wedge \ldots \wedge u^{n^{-}}\otimes
\Phi(u^{1}\wedge \ldots \wedge u^{n^{-}})\end{align}
for any orthonormal basis $\{u^{j}\}^{n^{-}}_{j=1}$ of $N^{-}_{y}$.
\end{defn}
It is easy to see that $J_{T}$ is an isometry from
$\textrm{F}^0$ onto its image.

For $\mu\geqslant 0, T >0$, let ${\bf E}^{\mu}_{T}$ be the image of
$\textrm{F}^{\mu}$ in $\textrm{\bf E}^{\mu}$ by $J_{T}$. Let ${\bf
E}^{0, \bot}_{T}$  be the orthogonal space to ${\bf E}^{0}_{T}$ in
${\bf E}^{0}$, and let $p_{T }, p^{\bot}_{T }$  be the orthogonal
projection operators from ${\bf E}^{0}$ on ${\bf E}^{0}_{T}, {\bf
E}^{0, \bot}_{T }$, respectively.

Recall that $\Lambda{(T^{\ast}M)}$ is  identified with
$\pi^{\ast}\big(\Lambda (T^{\ast}M)|_{B}\big)$ on ${\mathcal B}_{\varepsilon_{0}}\simeq {\mathcal
U}_{\varepsilon_{0}}$. Therefore if $s\in
\textrm{F}^{\mu}$, we can also consider $J_{T}s$ as an element of
$\textrm{E}^{\mu}$. Let ${\textrm E}^{\mu}_{T}$ be the image of
$\textrm{F}^{\mu}$ in $\textrm{E}^{\mu}$ by $J_{T}$. In particular,
${\bf E}_{T}^{0}$ may be identified isometrically with  $\textrm{E}^{0}_{T}$. Let
$\textrm{E}^{0, \bot}_{T}$ be the orthogonal space to
$\textrm{E}^{0}_{T}$ in $\textrm{E}^{0}$. Then $\textrm{E}^{0}$
splits orthogonally into
\begin{align}\label{2.51}
\textrm{E}^0=\textrm{E}^0_{T} \oplus \textrm{E}^{0,
\bot}_{T}.\end{align}
Let $\overline{p}_{T}, \overline{p}^{\bot}_{T
}$ be the orthogonal projection operators from $\textrm{E}^{0}$ on
$\textrm{E}^{0}_{T}, \textrm{E}^{0, \bot}_{T}$, respectively.
Since
$\textup{E}_{T}^{0}$ may be identified isometrically with ${\bf E}^{0}_{T}$,
we find that
\begin{align}\label{2.51a}
\overline{p}_{T}s=p_{T}s, \  \textup{for any} \ s\in {\textup{E}}^{0}\,,\
\operatorname{supp}(s)\subset {\mathcal B}_{\varepsilon_{0}}.
\end{align}
In particular,
\begin{align}\label{2.51b}
\overline{p}_{T}J_{T}s=p_{T}J_{T}s, \  \textup{for any} \ s\in {\textup{F}}^{0}.
\end{align}
According to the decomposition (\ref{2.51}) we set:
\begin{align}\begin{split}\label{2.52} D_{ T, 1}&= \overline{p}_{T}D_{ T}
\overline{p}_{T}, \ \ \  {D}_{T,
2}= \overline{p}_{T} D_{ T} \overline{p}^{\bot}_{T}, \\
D_{T, 3}&= \overline{p}^{\bot}_{T } D_{T}\overline{p}_{T}, \ \ \
D_{T, 4}= \overline{p}^{\bot}_{T}D_{T} \overline{p}^{\bot}_{T}.
\end{split}\end{align} Then
 \begin{align}\label{2.53}D_{T}= D_{T, 1}+ D_{T, 2}+
D_{T, 3}+ D_{T, 4}.\end{align}
We will now establish various estimates for the $D_{T, j}$'s as $
T\rightarrow +\infty$.
We define a twisted de Rham operator
\begin{align}\label{2.54}
D^{B}=\sum_{j=1}^{n}c(f_{j})\nabla_{f_{j}}^{B}: \Omega(B, o(N^{-}))
\rightarrow \Omega(B, o(N^{-})),
\end{align} where $\nabla^{B}=\nabla^{TB}\otimes 1+1\otimes \nabla^{o(N^{-})}$. The following
Lemma is similar to \cite[Th.\,9.8]{Bismut74} and \cite[Lemma 3.1]{Feng01}

\begin{prop} \label{t2.11} As $ T\rightarrow +\infty$, the
following formula holds
\begin{align}\label{2.55}J_{T}^{-1}D_{T,
1}J_{T}=D^{B}+O\big(\tfrac{1}{\sqrt{ T}}\big),\end{align} where
$O(\frac{1}{\sqrt{ T}})$ is a first order differential operator with
smooth coefficients dominated by $C/\sqrt{T}$.\end{prop}

\begin{proof} We can proceed as in \cite[Th.\,9.8]{Bismut74}, \cite[Lemma\,3.1]{Feng01}.
The proof becomes easier because of the simpler local
formula (\ref{2.42}) of the gradient of $f$.
By (\ref{2.31}),
 \begin{align}\label{2.56}
 D_{T, 1}=\overline{p}_{T}D_{ T}
\overline{p}_{T}=\overline{p}_{T}\big(D^{H}+D^{N}_{T}+R_{T}\big)\overline{p}_{T},
\end{align} where
\begin{align}\label{2.57}
R_{T}=O(|Z|^{2}\partial^{H}+|Z|^{2}\partial^{N}+|Z|+T|Z|^{4}).
\end{align}
From (\ref{2.51a}), (\ref{2.51b}) and (\ref{2.56}), we find that
\begin{align}\label{2.58}
J_{T}^{-1}D_{T, 1}J_{T}=
J_{T}^{-1}p_{T}\big(D^{H}+D^{N}_{T}+R_{T}\big)p_{T}J_{T}.
\end{align}
We may write out the projection $p_{T}$ explicitly.
From (\ref{2.50}), one verifies directly that for $s\in {\bf E}^0$,
\begin{align}\begin{split} \label{2.59}
p_{T}s(y, Z)=&\frac{1}{\alpha_{T}(y)}\rho(Z)
\exp\!\big(\!-\tfrac{T{|Z|}^2}{2}\big)  \\& \int_{N_{y}}\big\langle s(y, Z'),
\theta_{y}\big\rangle\rho(Z')\exp\!\big(\!-\tfrac{T{\left\vert
Z'\right\vert}^2}{2}\big)dv_{N_{y}}(Z')\wedge \theta_{y}.
\end{split}\end{align}
From (\ref{2.15}) and \cite[Prop.\,1.20]{Berline04},
we find
\begin{align}\label{2.60}
\nabla^{N}_{f^{H}_{j}}Z=0, \ \ \nabla^{N}\theta_{y}=0.
\end{align}
For $\ s\in \textrm{F}^{1}$, \eqref{2.60} yields
\begin{align}\begin{split}\label{2.61}
D^{H}J_{T}s(y, Z)=&\sum_{j=1}^{n}c(f_{j})\pi^{\ast}\nabla^{\Lambda (T^{\ast}M)|_{B}}_{f_{j}^{H}}
\left[\frac{1}{\sqrt{\alpha_{T}}}\,\rho(Z)\exp\!\big(\!-\tfrac{T{|Z|}^2}{2}\big)
 s(y)\wedge \theta_{y}\right]\\=&
\frac{1}{\sqrt{\alpha_{T}}}\,
\rho(Z)\exp\!\big(\!-\tfrac{T{|Z|}^2}{2}\big)
  \sum_{j=1}^{n}c(f_{j}) \pi^{\ast}\big(\nabla^{B}_{f_{j}}s(y)\big)\wedge \theta_{y}
\\=&J_{T}D^{B}s(y).
\end{split}\end{align}
For $s\in \textrm{F}^0$
\begin{align}\label{2.62} D_{T}^{N}J_{T}s
=\frac{(-1)^{|s|}}{\sqrt{\alpha_{T}}}\textrm{exp}(-\frac{T {|
Z|}^{2}}{2}) s(y)\wedge c\big(\nabla
\rho(Z)\big)\theta_{y},
\end{align} where $|s|$ denotes the degree of $s$ and $\nabla\rho(Z)$ is calculated in the
fiber direction, i.e.,
\begin{align}\label{2.63}
\nabla\rho(Z)=\sum_{\alpha=1}^{l}(e_{\alpha}\rho)(Z)e_{\alpha}.
\end{align}
From (\ref{2.59}), (\ref{2.62}) and (\ref{2.63}), we get that
\begin{align}\label{2.64}
p_{T}D_{T}^{N}p_{T}J_{T}s=0.\end{align}
For the term containing $R_{T}$, one verifies directly that when $ T\geqslant
1, \gamma\in \mathbb{R}, \ s\in {\bf E}^0,$
\begin{align}\label{2.65}\big\|p_{T}|Z|^{\gamma}s\big\|_{{\bf E}^0} \leqslant
\frac{C}{T^{\frac{\gamma}{2}}} \big\|s \big\|_{{\bf E}^{0}}.
\end{align}
Using (\ref{2.57}) and (\ref{2.65}), we find
\begin{align}\label{2.68}
J^{-1}_{T}p_{T}R_{T}p_{T}J_{T}=O\big(\tfrac{1}{\sqrt{T}}\big)\,,\quad T\rightarrow \infty.
\end{align}
Finally (\ref{2.61}), (\ref{2.64}) and (\ref{2.68}) imply the conclusion
of Proposition \ref{t2.11}.
\end{proof}
Set \begin{align*}{\textrm E}^{\mu, \bot}_{T}={\textrm E}^{\mu}\cap
{\textrm E}^{0, \bot}_{T}.\end{align*}
Similarly to the proof of Theorems 9.10, 9.11 and 9.14 from
\cite[\S9]{Bismut74}, we also have the following results.

\begin{lemma}\label{t2.13} There exists $T_{0}>0,\ C_{1}>0, C_{2}>0$ such that for any
$T\geqslant T_{0}, s\in \textup {E}^{1, \bot}_{T}, s_{1}\in
\textup{E}^{1}_{T}$, we have
\begin{align}\begin{split}\label{2.70}
\big\|{D}_{T,2}s\big\|_{\textup{E}^0} \leqslant& \frac{C_{1}}{\sqrt{T}}
\big\|s\big\|_{\textup{E}^1}, \\
\big\|{ D}_{ T, 3}s_{1}\|_{\textup{E}^0} \leqslant& \frac{C_{1}}{\sqrt
{T}} \|s_{1}\|_{{\textup
E}^1}, \\
\big\|D_{T, 4}s\big\|_{\textup{E}^0} \geqslant& C_{2}(\big\|s\big\|_{{\textup E}^1}+{\sqrt
{T}}\big\|s\big\|_{\textup{E}^{0}}).\end{split} \end{align} \end{lemma}

\subsection{Proof of (\ref{1.9}) and (\ref{1.10})}\label{s0.6}

In the first part of this section,
we prove Proposition \ref{t2.2} and then (\ref{1.9}) and (\ref{1.10}).
Let $C_{0}\in (0, 1]$ be a constant such that
\begin{align}\label{2.71}
\textup{Spec} \ (D^{B}) \cap [-2\sqrt{C_{0}},
2\sqrt{C_{0}}]\subset \{0\},
\end{align}
where $\textup{Spec}\ (D^{B})$ denotes the spectrum of the operator $D^{B}$.

Let $\mathcal{L}(\textrm{E}^{0})$ denote the space of all bounded linear operators from
$\textrm{E}^{0}$ into itself. For $A\in \mathcal{L}(\textrm{E}^{0})$ and $T\geqslant 1$, we write $A$ as
a matrix with respect to the splitting $\textrm{E}^{0}=\textrm{E}^{0}_{T}
\oplus \textrm{E}^{0,\bot}_{T}$ in the form
\begin{align}\label{2.72}
A=\begin{pmatrix}A_{1}& A_{2}
\\ A_{3}&A_{4}\end{pmatrix}.
\end{align}

\begin{defn}
For $A\in \mathcal{L}(\textrm{E}^{0}), P\in \mathcal{L}(\textrm{F}^{0})$, set
\begin{align}\label{2.73}
d(A,P)=\sum_{j=2}^{4}\big\|A_{j}\big\|_{1}+\big\|J^{-1}_{T}A_{1}J_{T}-P\big\|_{1},
\end{align} where the operator norm $\big\|\cdot\big\|_{1}$ is given by
$\big\|A\big\|_{1}=\textup{Tr}\big[(A^{\ast}A)^{\frac{1}{2}}\big]$. \end{defn}

Let $\textrm{F}_{T}^{C_{0}}$ be the direct sum of eigenspaces of $D^{2}_{T}$ associated to the
eigenvalues lying in $[0, C_{0})$. Let $P_{T}^{C_{0}}$ be the orthogonal
projection operator from $\textrm {E}^{0}$ on
$\textrm{F}_{T}^{C_{0}}$. Let $Q$ be the orthogonal projection from $\textrm {F}^{0}$ to $K=\textrm {Ker} \ D^{B}$. Similar to
\cite[(9.155)]{Bismut74},
we also have the following:

\begin{prop}\label{t2.29} For $T$ large enough, we have
\begin{align}\label{2.74}
d(P^{C_{0}}_{T}, Q)\leqslant \frac{C}{\sqrt{T}}.
\end{align}
\end{prop}

\begin{proof}[Proof of Proposition \ref{t2.2}]
From (\ref{2.74}), we see that for $T$ large enough,
\begin{align}\label{2.75}
\textup{dim} \ \textrm{F}_{T}^{C_{0}}=\textrm{dim} \ K.
\end{align} Let $P_{j}$ denote the orthogonal projection
operator from $\textrm{E}^{0}$ onto the $L^2$-completion
space of $\Omega^{j}(M)$ with respect to the metric
(\ref{2.24}). We need to show that when $T$ is large enough,
\begin{align}\label{2.76}
\textup{dim} \ P_{j} \big(\textup{F}_{T}^{C_{0}}\big)= q_{j}. \end{align}
By (\ref{2.75}), we find that
\begin{align} \label{2.77}
\sum_{j=0}^{m}\textrm{dim} \ P_{j}
\big(\textrm{F}_{T}^{C_{0}}\big)\leqslant \textrm{dim} \
\textrm{F}_{T}^{C_{0}}=\sum_{j=0}^{m}q_{j}.
\end{align}
Also, we find that for any $s_{j}\in F_{j},
\big\|s_{j}\big\|_{\textrm{F}^{0}}=1$,
\begin{align}\begin{split}\label{2.771}
\big\|P_{j}P_{T}^{C_{0}}J_{T}s_{j}-J_{T}s_{j}\big\|_{{\bf
E}^{0}} \leqslant
d(P_{T}^{C_{0}}, Q).
\end{split}\end{align} Thus from (\ref{2.74}) and (\ref{2.771}),  we have
for
 $s\in K$,
\begin{align}\label{2.78}
\big\|P_{j}P_{T}^{C_{0}}J_{T}s-J_{T}s\big\|_{{\bf E}^{0}}\leqslant
\frac{C}{\sqrt{T}}\big\|s\big\|_{\textup{F}^{0}}. \end{align}
From (\ref{2.78}), one deduces that for sufficiently large $T$,
\begin{align}\label{2.79}\textrm{dim} \ P_{j}
\big(\textrm{F}_{T}^{C_{0}}\big)\geqslant q_{j}.
\end{align}  From (\ref{2.77}) and (\ref{2.79}),
we get (\ref{2.76}). Since the action of $G$ commutes with the deformed de Rham operator
$D_{T}$, the eigenspaces of $D^{2}_{T}$ with eigenvalues in $[0, C_{0})$ are $G$-vector spaces.
This completes the proof of Proposition \ref{t2.2}.
\end{proof}

Let $\textrm{F}^{C_{0}}_{T, j}$ denotes the $q_{j}$-dimensional vector
space generated by the eigenspaces of $D^2_{T}|_{\Omega^j(M)}$
associated with the eigenvalues lying in $[0, C_{0}), j=0, 1, \ldots, m$.
Then $G$ maps $\textrm{F}^{C_{0}}_{T, j}$ into itdelf.

Recall that the isometric map $J_{T}:
\textrm{F}^{0}\rightarrow \textrm{E}^{0}$ is defined by (\ref{2.50}).
We define the map  $e_{T}: \textrm{F}^{0}\rightarrow \textrm{F}^{C_{0}}_{T}$ by
$e_{T}=P^{C_{0}}_{T}J_{T}$. We will prove that $e_{T}$ is
a $G$-isomorphism from  $F_{j}$ onto its image when $T$ is large enough.

\begin{lemma}\label{t3.1}
There exists $C>0$ such that for any $s\in F_{j}$,
\begin{align}\label{3.1}\big\|(e_{T}-J_{T})s\big\|_{\textup{E}^{0}}
=O\big(\tfrac{C}{\sqrt{T}}\big)\big\|s\big\|_{\textrm{F}^{0}} \
\ \ \ \text{as $T\rightarrow +\infty$}\,. \end{align}  In particular,
$e_{T}$ is an $G$-isomorphism from $F_{j}$ onto $\mathrm{F}^{C_{0}}_{T,j}$.
\end{lemma}

\begin{proof} It is clear that $e_{T}$ maps $F_{j}$ into
$\textrm{F}^{C_{0}}_{T,j}$ and
\begin{align}\label{3.2}(e_{T}-J_{T})s=p_{T}P_{T}^{
C_{0}}J_{T}s-J_{T}s+p^{\bot}_{T}P_{T}^{C_{0}}J_{T}s.
\end{align}
By Proposition \ref{t2.29}, for any $s\in F_{j}$,
\begin{align}\begin{split}\label{3.3}
\big\|(e_{T}-J_{T})s\big\|_{\textrm{E}^{0}} &\leqslant \big\|\big(
P_{T}^{C_{0}}\big)_{1}J_{T}s-J_{T}s\big\|_{\textrm{E}^{0}}
+\big\|\big(P_{T}^{C_{0}}\big)_{3}J_{T}s\big\|_{\textrm{E}^{0}}
\\ &\leqslant \frac{C}{\sqrt{T}}\big\|s\big\|_{\textrm{F}^{0}}.\end{split}\end{align}
Therefore,  $e_{T}|_{F_{j}}$ is injective for $T$ large enough.
Moreover, \begin{align}\label{3.5}\textup{dim} \ F_{j}=\textrm{dim} \
F^{C_{0}}_{T,j}=q_{j}.\end{align} Thus $e_{T}$ is an isomorphism
from $F_{j}$ onto $\textrm{F}_{T, j}^{C_{0}}$\,.
Since $g^{N}$ is $G$-invariant,
\begin{align}\label{3.5a}
\big|g^{-1}\cdot Z\big|_{g_{g^{-1}\cdot y}^{N}}=\big|Z\big|_{g_{y}^{N}},\ \ g\cdot \theta=\theta.
\end{align}
From (\ref{2.48}) we have $\alpha_{T}(y)=\alpha_{T}(g^{-1}\cdot y)$.
From (\ref{2.50}) and (\ref{3.5a}), we find that for any $s\in \textup{F}^{0}$,
\begin{align}
\big(g\cdot J_{T}s\big)(y, Z)=&g\cdot (J_{T}s)(g^{-1}\cdot y, g^{-1}\cdot Z)
 \nonumber \\=&
  \frac{1}{\sqrt{\alpha_{T}(g^{-1}\cdot y)}}
  \rho\Big(\big|g^{-1}\cdot Z\big|_{g_{g^{-1}\cdot y}^{N}}\Big)\textrm{exp}\Big(-\frac{T
  \big|g^{-1}\cdot Z\big|_{g_{g^{-1}\cdot y}^{N}}^{2}}{2}\Big)\times
  \nonumber \\&
   g\cdot s(g^{-1} \cdot y)\wedge g\cdot \theta_{g^{-1}\cdot y}
  \\=&
  \frac{1}{\sqrt{\alpha_{T}(y)}} \rho\Big(|Z|_{g_{y}^{N}}\Big)\textrm{exp}\Big(-\frac{T
  |Z|_{g_{y}^{N}}^{2}}{2}\Big)  \big(g\cdot s\big)(y)\wedge (g\cdot \theta)_{y}
  \nonumber \\=&J_{T}\big(g\cdot s\big)(y, Z).\nonumber
\end{align}
This shows that $g$ commutes with $J_{T}$.
Since $g$ commutes with $D_{T}$, $g$ commutes with $P^{C_{0}}_{T}$.
Therefore, $e_{T}$ is a $G$-map, i.e., it commutes with the action of $G$.
The proof of Lemma \ref{t3.1} is complete. \end{proof}

\begin{proof}[Proof of (\ref{1.9}) and (\ref{1.10})]
As $G$-representation space, $\mathrm{F}^{C_{0}}_{T,j}$ can be
decomposed as:
\begin{align}\label{3.6}
\mathrm{F}^{C_{0}}_{T,j}=\sum^{l_{0}}_{\alpha=1} \textrm{Hom}_{G}\big(V^{\alpha},
\mathrm{F}^{C_{0}}_{T,j}\big)\otimes V^{\alpha}. \end{align}
Then $\big(\textrm{Hom}_{G}(V^{\alpha}, \mathrm{F}^{C_{0}}_{T,\bullet}), d_{T}\big)$
is a $G$-subcomplex of the complex
$(\mathrm{F}^{C_{0}}_{T,\bullet}, d_{T})$. From (\ref{1.8}) and Lemma \ref{t3.1}, we see
that
\begin{align}\label{3.7}
\textrm{dim}\ \textrm{Hom}_{G}(V^{\alpha},
\mathrm{F}^{C_{0}}_{T,j})=d^{\alpha}_{j}.
\end{align}
From the Hodge theorem for complexes of finite-dimensional vector spaces, we know
that the $j$-th cohomology group of the complex $(\mathrm{F}^{C_{0}}_{T,
\bullet}\,, d_{T})$ is isomorphic to $\textup{Ker}
D^{2}_{T}|_{\Omega^{j}(M)}$. Thus
the dimension of the $j$-th cohomology group associated to the
complex $(\textrm{Hom}_{G}(V^{\alpha}, \mathrm{F}^{C_{0}}_{T, \bullet})\,, d_{T})$ is $b^{\alpha}_{j}$.
Then the inequalities (\ref{1.9}) and (\ref{1.10}) hold by standard algebraic techniques
\cite[Lemma 3.2.12]{Ma07}.
This completes the proof of Theorem \ref{t1.1}. \end{proof}

 \section{An application of equivariant Morse inequalities}
 \label{}
 In this Section, we apply the equivariant Morse inequalities to
prove Theorem \ref{t1.2}, which is a generalization of
\cite[Th.\,1]{Zadeh10}. We first prove a crucial Lemma.

\begin{lemma}\label{t4.1}
The following inequalities holds for $k=0,1,\ldots,m,$
\begin{align}\label{4.1d}
\sum_{j=0}^{k}(-1)^{k-j}\beta_{j}(M,N_{+})\leqslant \sum_{j=0}^{k}(-1)^{k}q_{j}.
\end{align}
The equality holds in (\ref{4.1d}) for $k=m$.
\end{lemma}

\begin{proof}
Set \begin{align}\label{4.3} M_{1}=M\cup_{N_{+}}(-M),  \ \
M_{2}=M_{1}\cup_{N^{\prime}_{-}}(-M_{1}),\end{align} where
$(-M), (-M_{1})$ are copies of $M$ and $M_{1}$, respectively, and
$N^{\prime}_{-}$ is the boundary of $M_{1}$, i.e.,
$N_{-}^{\prime}=N_{-}\sqcup (-N_{-})$.

Denote by $\mathbb{R}^{+}, \mathbb{R}^{-}$ the trivial and the nontrivial one-dimensional
real $\mathbb{Z}_{2}$-representation, respectively.
It is well-known that, as $\mathbb{Z}_{2}$-representation spaces,
$H^{j}(M_{2})$ and $H^{j}(M_{1})$ have the following
decompositions: 
\begin{align}\begin{split}\label{4.4}
{H}^{j}(M_{2})&=H^{j}(M_{1})\cdot \mathbb{R}^{+}\oplus H^{j}(M_{1},
N^{\prime}_{-})\cdot \mathbb{R}^{-},  \\ {H}^{j}(M_{1})&=H^{j}(M)\cdot
\mathbb{R}^{+}\oplus H^{j}(M, N_{+})\cdot \mathbb{R}^{-}.  \end{split}\end{align}
Let $\tau_{1}$ and $\tau_{2}$ be the flip maps of
$M_{1}$ and $M_{2}$, respectively. Let $e$ and $g$ be the
trivial and nontrivial element, respectively, in $\mathbb{Z}_{2}$, which
can be viewed as a multiplication group, i.e., $g^{2}=e=e^{2}$.
Then $\mathbb{Z}_{2}\times \mathbb{Z}_{2}$ acts naturally
on $M_{2}$ by
\begin{align}\label{4.5}
(e, g)\cdot x=\tau_{1}(x),   \ \ \ (g, e)\cdot x=\tau_{2}(x),\ \ \
 \  \forall \ x\in M_{2}. \end{align}
Let $\{W^{\alpha}\}^{4}_{\alpha=1}$  be the set of non-isomorphic
irreducible representations of $\mathbb{Z}_{2}\times \mathbb{Z}_{2}$. As vector space, $W^{j}=\mathbb{R}$
but $(e, g)$ acts as $\textup{Id}$ on $W^{1}, W^{2}$ and acts as $-\textup{Id}$ on $W^{3}, W^{4}$;
besides $(g, e)$ acts as $\textup{Id}$ on $W^{1}, W^{3}$ and acts as $-\textup{Id}$ on $W^{2}, W^{4}$.

Recall that $b^{\alpha}_{j}$ and $d^{\alpha}_{j}$ are defined  in (\ref{1.8}).
Using (\ref{4.4}) and the Poincar\'e duality theorem
for manifolds with boundary
\cite[Chap.\,5,\,Prop.\,9.12]{Taylor96}, we calculate directly,
\begin{align}  \label{4.7} b^{1}_{j}=\beta_{j}(M), \ b^{2}_{j}=\beta_{m-j}(M),
\ b^{3}_{j}=\beta_{j}(M, N_{+}), \ b^{4}_{j}=\beta_{m-j}(M, N_{+}).
\end{align}
To clarify our statement, we replace
the vector space $F_{j}, F_{a+, j}$ and $F_{r-, j}$
by $F_{j}(f), F_{a+, j}(f)$ and $F_{r-, j}(f)$, respectively.
The Poincar\'e duality theorem yields
\begin{align}\label{4.7a}
F_{j}(-f)\simeq F_{m-j}(f).
\end{align}
By considering the Morse-Bott function $-f$ instead of $f$, we obtain that
\begin{align}\begin{split} \label{4.8} d^{1}_{j}=q_{m-j}+\textup{dim}_{\mathbb{R}}
F_{r-, j-1}(-f)+\textup{dim}_{\mathbb{R}}
F_{a+, j}(-f)
, & \ \ d^{2}_{j}=q_{m-j}+\textup{dim}_{\mathbb{R}}
F_{r-, j-1}(-f);
  \\ d^{3}_{j}=q_{m-j}+\textup{dim}_{\mathbb{R}}
F_{a+, j}(-f), & \ \ d^{4}_{j}=q_{m-j}.
  \end{split} \end{align}
Now we  prove (\ref{4.8}) as follows.
For $w\in F_{j}(-f)$, let $\big\{w\big\}$ be the real line generated by $w$. Set
\begin{align}\label{4.9}
W=\big\{w\big\}\oplus\big\{\tau_{1}(w)\big\}\oplus
\big\{\tau_{2}(w)\big\}\oplus
\big\{\tau_{1}\tau_{2}(w)\big\}. \end{align} Then $W$ is a 4-dimensional vector space spanned by
$\big\{w\big\}, \big\{\tau_{1}(w)\big\}, \big\{\tau_{2}(w)\big\}$ and $\big\{\tau_{1}\tau_{2}(w)\big\}$.
Note that $\mathbb{Z}_{2}\times \mathbb{Z}_{2}$ acts naturally on $W$,
which can be rewritten as
\begin{align}\begin{split}\label{4.10} W=&
\big\{w+\tau_{1}(w)+\tau_{2}(w)+\tau_{2}\tau_{1}(w)\big\} \oplus
\big\{w+\tau_{1}(w)-\tau_{2}(w)-\tau_{2}\tau_{1}(w)\big\}
\\&\oplus \big\{w-\tau_{1}(w)+\tau_{2}(w)-\tau_{2}\tau_{1}(w)\big\} \oplus
\big\{w-\tau_{1}(w)-\tau_{2}(w)+\tau_{2}\tau_{1}(w)\big\}.
\end{split}\end{align}
Moreover, the 1-dimensional $\mathbb{Z}_{2}\times\mathbb{Z}_{2}$-space
$\big\{w+\tau_{1}(w)+\tau_{2}(w)
+\tau_{2}\tau_{1}(w)\big\}$ $\Big($resp.
$\big\{w+\tau_{1}(w)-\tau_{2}(w)
-\tau_{2}\tau_{1}(w)\big\}$, resp.
$\big\{w-\tau_{1}(w)+\tau_{2}(w)
-\tau_{2}\tau_{1}(w)\big\}$, resp.
$\big\{w-\tau_{1}(w)
-\tau_{2}(w)+\tau_{2}\tau_{1}(w)\big\}\Big)$ is
isomorphic to $W^{1}$ (resp $W^{2}$, resp. $W^{3}$, resp. $W^{4}$)
as $\mathbb{Z}_{2}\times\mathbb{Z}_{2}$ representation space.
Thus as a $\mathbb{Z}_{2}\times \mathbb{Z}_{2}$
representation space,
\begin{align}\label{4.11}
W\simeq W^{1}\oplus W^{2}\oplus W^{3}\oplus W^{4}.
\end{align}

For nondegenerate critical manifolds on the boundary, we have for $w\in F_{r-, j-1}(-f), w'\in F_{a+, j}(-f)$,
\begin{align}\begin{split}\label{4.12}
\big\{w \big\}\oplus \big\{\tau_{2}(w)\big\}
&=\big\{w+\tau_{2}(w)\big\}
\oplus
\big\{w-\tau_{2}(w)\big\};
\\ \big\{w'\big\}\oplus \big\{\tau_{1}(w')\big\}&=
\big\{w'+\tau_{1}(w')\big\}\oplus
\big\{w'-\tau_{1}(w')\big\}.
\end{split}\end{align}
From (\ref{4.7a}), (\ref{4.11}) and (\ref{4.12}), one get (\ref{4.8}) immediately.

Applying the equivariant Morse inequalities (\ref{1.9}) to
$\alpha=4$, we deduce that
\begin{align}\label{4.13}
\sum_{j=0}^{k}(-1)^{k-j}\beta_{m-j}(M,N_{+})\leqslant \sum_{j=0}^{k}(-1)^{k-j}q_{m-j}.
\end{align} From (\ref{4.13}), one verifies directly that
\begin{align}\label{4.14}
\sum_{j=0}^{k}(-1)^{k-j}\beta_{j}(M,N_{+})\leqslant \sum_{j=0}^{k}(-1)^{k-j}q_{j}.
\end{align}
One verifies easily that the equality in (\ref{4.14}) holds when $k=m$.
The proof of Lemma \ref{t4.1} is complete.
\end{proof}

\begin{proof}[Proof of (\ref{1.13})]

We now consider the Mayer-Vietories sequence \cite[pp.\,185]{Massey} associated with the
triad $(M, N_{+}, N_{r+})$:
\begin{align}\label{4.15}
\ldots\rightarrow {H}^{j-1}(N_{a+})
\rightarrow {H}^{j}(M, N_{+}) \rightarrow {H}^{j}(M, N_{r+})
\rightarrow {H}^{j}(N_{a+}) \rightarrow \ldots
\end{align}
From (\ref{4.15}) and \cite[Lemma 3.2.12]{Ma07}, we get
\begin{align}\label{4.16}\sum^{k}_{j=0}(-1)^{k-j}[\beta_{j}(N_{a+})-\beta_{j}(M, N_{r+})+\beta_{j}(M, N_{+})]
=\textrm{dim Im}\ \delta_{1}^{k}, \end{align} where
$\delta_{1}^{k}$ denotes the connecting morphism  ${H}^{k}(N_{a+})\rightarrow
{H}^{k+1}(M, N_{+})$ in the long exact sequence (\ref{4.15}).

Next we consider the triad $(M, N_{r}, N_{r+})$:
\begin{align}\label{4.17}\ldots\rightarrow {H}^{j-1}(N_{r-}) \rightarrow
{H}^{j}(M, N_{r}) \rightarrow {H}^{j}(M, N_{r+}) \rightarrow
{H}^{j}(N_{r-}) \rightarrow \ldots \end{align}
From (\ref{4.17}) and \cite[Lemma 3.2.12]{Ma07}, we find that
\begin{align}\label{4.18}\sum^{k}_{j=0}(-1)^{k-j}
\big[\beta_{j}(M, N_{r+})-\beta_{j}(M,
N_{r})+\beta_{j-1}(N_{r-})\big] =\textrm{dim Im}\
\delta_{2}^{k},\end{align} where $\delta_{2}^{k}$ denotes the morphism
${H}^{k}(M, N_{r+})\rightarrow
{H}^{k}( N_{r-})$ in the long exact sequence (\ref{4.17}) induced by the inclusion
$(N_{r}, N_{r+}) \hookrightarrow (M, N_{r+}).$

From (\ref{4.14}), (\ref{4.16}) and (\ref{4.18}), we get that for $k=0,1,\ldots,m,$
\begin{align}\label{4.1}
\sum_{j=0}^{k}(-1)^{k-j}\beta_{j}(M,N_{r})\leqslant \sum_{j=0}^{k}(-1)^{k}\nu_{j},
\end{align} where
\begin{align}\label{4.2}
\nu_{j}=q_{j}+\beta_{j}(N_{a+})+\beta_{j-1}(N_{r-}).
\end{align}
The equality holds in (\ref{4.1}) for $k=m$.

We now directly apply the degenerate Morse inequalities \cite[(2.101)]{Bismut86} to
the closed submanifolds $N_{a+}$ and $N_{r-}$ respectively:  for
$0\leqslant k \leqslant m-1,$
\begin{align}\label{4.22}\sum_{j=0}^{k}(-1)^{k-j}\beta_{j}(N_{a+})\leqslant
\sum_{j=0}^{k}(-1)^{k-j}q_{a+, j}, \ \ \sum_{j=0}^{k}(-1)^{k-j}\beta_{j}(N_{r-})\leqslant
\sum_{j=0}^{k}(-1)^{k-j}q_{r-, j}\,, \end{align} with equality
for $k=m-1$. Note that
\begin{align}\label{4.23}\beta_{m}(N_{a+})=0=q_{a+,m}, \ \ \beta_{m}(N_{r-})=0=q_{r-,m}. \end{align}
Due to
(\ref{4.22}) and (\ref{4.23}),
the equality in (\ref{4.22}) holds also for $k=m$.

Now (\ref{1.13}) follows from (\ref{4.1}) and (\ref{4.22}). One verifies easily that
the equality in (\ref{1.13}) holds when $k=m$.
This finishes the proof of Theorem \ref{t1.2}. \end{proof}

\noindent
\textbf{\emph{Acknowledgements.}}
The author would like to thank Professors Xiaonan Ma and George Marinescu for their kind advices.

\end{document}